\documentclass[12pt]{article}
\usepackage{geometry}
\usepackage{amsfonts}
\usepackage{mathptmx}

\usepackage[french, english]{babel}
\usepackage[
backend=biber,
style=numeric,
sorting=nyt
]{biblatex}

\usepackage{amsmath, amsthm, amsfonts,amssymb,graphicx,csquotes,mathtools}
\usepackage{authblk}
\usepackage{float}
\usepackage{bigints}
\usepackage{tikz-cd}

\newtheorem{theorem}{Theorem}
\newtheorem{proposition}[theorem]{Proposition}
\newtheorem{lemma}[theorem]{Lemma}

\theoremstyle{definition}
\newtheorem{definition}[theorem]{Definition}

\theoremstyle{remark}
\newtheorem{remark}{Remark}[section]

\title{$C^1$-genericity of unbounded distortion for ergodic conservative expanding circle maps.}
\author{Hamza Ounesli$^{1,2}$}

\date{%
    $^1$\small\textit{Scuola Internazionale Superiorie di Studi Avanzati (SISSA), Trieste, Italy.} \vspace{0.2cm}%
   \\ $^2$\textit{Abdus Salam International Centre for Theoretical Physics (ICTP), Trieste, Italy.}\\(\textit{Email: hounesli@sissa.it and hounesli@ictp.it})\\[2ex]%
    \today
}

\begin{document}
\maketitle
\begin{abstract}
We prove that within the space of ergodic Lebesgue-preserving $C^1$ expanding maps of the circle, unbounded distortion is $C^1$-generic.
\end{abstract}
%\selectlanguage{french} 
%\begin{abstract}
 %  On d\'emontre que $C^1$-g\'en\'eriquement les transformation expansive du cercle qui sont conservative et ergodique jouissent la propri\'et\'e de distortion non born\'ee. Ce fait surprenant va \`a l'encontre de ce qui est attendu compte tenu des r\'esultats classiques de la th\'eorie des syst\`emes dynamique hyperbolique.
%\end{abstract}
\selectlanguage{english}

\section{Introduction and statement of results.}
Let $E^1(\mathbb S^1)$ be the space of $C^1$ orientation-preserving uniformly expanding  maps on the circle, i.e. \( C^{1}\) maps \( f: \mathbb S^{1}\to \mathbb S^{1}\)  for which   there exists some uniform constant $\sigma>1$ such that $f'(x)\geq \sigma$, We let \( \lambda\) denote Lebesgue measure on \( \mathbb S^{1}\) and recall that a map \( f: \mathbb S^{1}\to \mathbb S^{1}\) is said to preserve Lebesgue measure if   \( \lambda (f^{-1}(A))= \lambda (A)\) for any measurable set \( A \subseteq \mathbb S^{1}\).  We let 
\begin{equation*}
    \Gamma_{\lambda}(\mathbb{S}^1) \coloneqq \lbrace f\in E^{1}(\mathbb{S}^1): f\ \text{preserves Lebesgue measure} \rbrace.
\end{equation*}
 The simplest and well known examples of  maps in $ \Gamma_{\lambda}(\mathbb{S}^1)$ are the maps of the form \( f(x) = \kappa x \) mod 1, for \( \kappa \in \mathbb N, \kappa\geq 2 \). These maps are affine and are therefore easily seen to preserve Lebesgue measure \( \lambda\). Although it is not immediately intuitive, it turns out there are many other maps in $E^1(\mathbb S^1)$ which also preserve Lebesgue measure without being piecewise affine. The following result makes explicit the remarkable flexibility in their construction. 

We first introduce some notation. For any \( n \geq 2 \), let  $\lbrace I_{i}\rbrace_{1\leq i\leq n}$ be a partition of the unit interval into non-trivial adjacent closed subintervals \( I_{i}=[x_{i}^{-}, x_{i}^{+}] \) which intersect only at the endpoints, so that \( x_{0}^{-}=0\) and \( x_{n}^{+}=1\).  Let $1\leq i_{0}\leq n$ and suppose that for each $i\neq i_{0}$ there is given a \( C^{1}\)  expanding diffeomorphism $f_{i}: I_{i}\to [0,1]$. We are interested in whether we can construct  a \( C^{1}\)  expanding diffeomorphism $f_{i_{0}}: I_{i_{0}}\to [0,1]$ which defines a piecewise \( C^{1}\) full branch map preserving Lebesgue measure, and further more whether this can actually be constructed in such a way as to represent a a \( C^{1}\) uniformly expanding circle map. The following result give some natural and explicitly verifiable necessary and sufficient conditions for this to be the case. 

\begin{proposition}[Missing Branch Extension Proposition]\label{prop:extension}
There exists a (unique) extension to a  Lebesgue-preserving uniformly expanding full branch map  if and only if  for all \( x \in [0,1]\) we have:
\begin{equation}\label{eq:condition 1}
    \sum_{\substack{1\leq i\leq n\\ i\neq i_{0}}}\dfrac{1}{f'_{i}\circ f_{i}^{-1}(x)}<1 
\end{equation}
Moreover, this extension represents a \( C^{1}\) uniformly expanding circle map of degree \( n \) if and only if 
\begin{equation}\label{eq:matching}
f'_{i-1}(x_{i-1}^{+})=f'_{i}(x_{i}^{-})
\quad \text{ and } \quad 
 f'_{i+1}(x_{i+1}^{-}) = f'_{i}(x_{i}^{+})
\end{equation}
for all \( i\neq i_{0}\) (if \( i =0 \) we  replace \( i-1\) by \( n \) in the first equality, and if \( i=n\) we replace \( i+1\) by 1 in the second inequality),  
and 
\begin{equation}\label{eq:The condition to extend a full branch to an expanding map}
f'_{i_{0}-1}(x_{i_{0}-1}^+)=\dfrac{1}{1-\sum_{\substack{i\neq i_{0}}}f'_{i}(x_{i}^{-})}
\end{equation}
and 
\begin{equation}\label{eq:The condition to extend a full branch to an expanding map2}
f'_{i_{0}+1}(x_{i_{0}+1}^-)=\dfrac{1}{1-\sum_{\substack{i\neq i_{0}}}f'_{i}(x_{i}^{+})}
\end{equation}
(where also, if  \( i_{0}=1\) we replace \( i_{0}-1\) by \( n \) in the left hand side of  \eqref{eq:The condition to extend a full branch to an expanding map} and if \( i_{0}=n\) we replace \( i_{0}+1\) by  1 in the left hand side of \eqref{eq:The condition to extend a full branch to an expanding map2}).
\end{proposition}

\begin{remark}
We remark that \eqref{eq:The condition to extend a full branch to an expanding map} is a very mild and very ``open'' condition that allows a huge amount of flexibility in the choice of the given branches, thus indicating that the space of such Lebesgue preserving maps is very large. Condition  \eqref{eq:matching} gives essentially trivial matching conditions on the derivatives at the left and right endpoints of the domains of each branch which ensure that the map is \( C^{1}\) when considered as a circle map.  Conditions \eqref{eq:The condition to extend a full branch to an expanding map} and \eqref{eq:The condition to extend a full branch to an expanding map2} are less intuitively immediate but essentially ensure that the unique extension given in the first part of the proposition also satisfies such matching conditions. 
\end{remark}

\begin{remark}
For \( n =2 \), i.e. circle maps of degree 2, this result was proved in  \cite{9},  
where it was used to show that the space of degree 2 uniformly expanding circle maps which preserve Lebesgue measure is homeomorphic to the infinite-dimensional Lie group \( \mathbb T^{2}\setminus \{diag \ \mathbb T^{2}\} \times {\mathfrak D}^{1} (\mathbb S^{1}, 0)\), where \( \mathbb T^{2}\)  is the two-dimensional torus and \( {\mathfrak D}^{1} (\mathbb S^{1}, 0)\) is the space of \( C^{1}\) diffeomorphisms of \( \mathbb S^{1}\) which has \( 0 \) as a fixed point, endowed with the \( C^{1}\) topology. The space \(     \Gamma_{\lambda}(\mathbb{S}^1) \) includes maps of arbitrary degree and thus, in particular, all degree 2 maps considered in \cite{9} and therefore strictly contains in particular a space homeomorphic to \( \mathbb T^{2}\setminus \{diag \ \mathbb T^{2}\} \times {\mathfrak D}^{1} (\mathbb S^{1}, 0)\).  
\end{remark}

Theorem \ref{prop:extension} does not tell us anything about the \emph{ergodicity} of Lebesgue measure. Indeed, remarkably, there exist examples of maps in  $\Gamma_{\lambda}(\mathbb{S}^{1})$ for which \emph{Lebesgue measure is invariant but not ergodic} \cite{13} and therefore 
\begin{equation*}
    \Gamma_{\lambda,er}(\mathbb{S}^1)=\lbrace f\in \Gamma_{\lambda}(\mathbb{S}^1): \lambda  \text{ is ergodic} \rbrace \text{ is strictly contained in }     \Gamma_{\lambda}(\mathbb{S}^1).
\end{equation*}
On the other hand, it was proved in   \cite{2} that  $\Gamma_{\lambda,er}(\mathbb{S}^{1})$ is \emph{residual} in  \(  \Gamma_{\lambda}(\mathbb{S}^1) \) in the \( C^{1}\) topology,  and so,  \emph{generically}, Lebesgue measure is ergodic whenever it is invariant. thus implying that $\Gamma_{\lambda,er}(\mathbb{S}^{1})$  is also a relatively large set. One of the main tools used to prove ergodicity in specific classes of uniformly expanding circle or full branch maps  is  \emph{bounded distortion}. We recall that every   $f\in E^{1}(\mathbb S^1)$ admits a family of partitions \( \mathcal P_{n}\coloneqq \{\omega_{ n, i}\}\)  which are  injectivity domains of $f^{n}$,  and say that \( f \) has \emph{bounded distortion} if 
\begin{equation*}
    \mathcal{D} \coloneqq \sup_{n\geq 1}  \sup_{\omega_{n,i}\in \mathcal P_{n}}\sup_{x,y\in \omega_{n, i}}
    \log\dfrac{|(f^n)'(x)|}{|(f^n)'(y)|} < \infty.
\end{equation*}
It is well known, by standard arguments, that bounded distortion implies ergodicity of Lebesgue measure (even if Lebesgue measure is not invariant). The converse is however not true and there are maps for which Lebesgue measure is ergodic but which do not satisfy bounded distortion. 
We then let 
\begin{equation*}
    \mathcal{B}=\lbrace f\in \Gamma_{\lambda, er}(\mathbb{S}^1):  f\ \text{has bounded distortion}\rbrace
\   \text{ and } \ 
     \mathcal B^{c} \coloneqq  \Gamma_{\lambda, er}(\mathbb{S}^1)\setminus \mathcal B 
\end{equation*}
be the two complementary sets of \( \Gamma_{\lambda, er}(\mathbb{S}^1)\) which respectively satisfy and do not satisfy the bounded distortion property. A natural question concerns the relative sizes of the sets \(    \mathcal{B} \) and \(   \mathcal B^{c} \) in  \( \Gamma_{\lambda}(\mathbb{S}^1)\).  By standard arguments it can be shown that if \( f \) is \(C^{1+\alpha}\), i.e if the derivative \( f'\) is \emph{H\"older continuous}, then \( f \) has bounded distortion, and it was proved in \cite{4}. that bounded distortion also holds whenever the derivative \( f'\) has a modulus of continuity which is \emph{Dini-integrable}, which is a strictly weaker condition that being H\"older continuous. Thus all maps with such regularity, which also preserve Lebesgue measure, belong to \( \mathcal B\). However, recently in \cite{6} we have shown that many maps which preserve the Lebesgue measure do not have the bounded distortion property and here we extent this observation to show that  in fact \( C^{1}\) \emph{generic maps in $ \Gamma_{\lambda,er}(\mathbb{S}^1)$ do not have bounded distortion}. 

\begin{theorem}\label{thm:main}
The set $\mathcal{B}^{c}$ is $C^1$ residual in $ \Gamma_{\lambda,er}(\mathbb{S}^1)$.
\end{theorem}

This results suggests that other techniques, which do not rely on bounded distortion,  need to be developed to prove the ergodicity of Lebesgue measure in specific classes of maps.

\section{Overview of the proof}
In Section \ref {sec:extension} we prove the Missing Branch Extension Proposition \ref{prop:extension} which is the key ingredient in the proof of the following result.  Consider the set \(  \Gamma_{\lambda}(\mathbb{S}^1) \)  of  uniformly expanding \( C^{1}\) circle maps which preserve Lebesgue measure but for which Lebesgue measure is not necessarily ergodic. Letting 
\[
 \mathcal{B_{\lambda}}=\lbrace f\in \Gamma_{\lambda}(\mathbb{S}^1):  f\ \text{has bounded distortion}\rbrace
\  \text{ and } \ 
     \mathcal B_{\lambda}^{c} \coloneqq  \Gamma_{\lambda}(\mathbb{S}^1)\setminus \mathcal B 
\]
we will prove the following in  Section \ref{sec:dense}. 
\begin{proposition}\label{dense}
\(  \mathcal{B_{\lambda}} \) is \( C^{1}\) dense  in $\Gamma_{\lambda}(\mathbb{S}^1)$.
\end{proposition}

Then,  in Section \ref{sec:residual} we will prove the following. 

\begin{proposition}\label{residual}
    The space $\Gamma_{\lambda}(\mathbb{S}^1)$ is \( C ^{1}\) residual in its completion $\Gamma^{\star}_{\lambda}(\mathbb{S}^1)$.
\end{proposition}

Now we will give a proof of the Theorem assuming the previous propositions.

\begin{proof}[Proof of Theorem \ref{thm:main}]
Let us denote by $d_{k}$ the map:
\begin{equation}
    d_{k}:\Gamma_{\lambda}(\mathbb{S}^1)\to\mathbb{R}_{+}
\end{equation}
defined by 
\begin{equation*}
    d_{k}(f)=\sup_{\substack{x,y\in\omega^{k}_{i}\\ 1\leq i\leq deg(f)^k}} |\log\dfrac{f^{k}(x)}{f^{k}(y)}|.
\end{equation*}
For every $k\in \mathbb{N}$, $d_k$ is continuous in the $C^1$ topology, to see that, first notice that for $\epsilon>0$ small enough and $f,g\in\Gamma_{\lambda}(\mathbb{S}^{1})$ such that $d(f,g)\leq \epsilon$, where $d$ denotes the $C^1$-distance, then $deg(f)=deg(g)$. On the other hand $k$ being fixed, $d(f^{k},g^{k})\leq C_{k}\epsilon$ where $C_{k}$ is a positive constant depending only on $k$. This two remarks are enough to conclude the continuity of $d_{k}$.

    Now notice that the complement of $\mathcal{UB}$ in $\Gamma_{\lambda}(\mathbb{S}^{1})$ is equal to the following set:

\begin{equation}
    \bigcup\limits_{n\in \mathbb{N}}\bigcup\limits_{m\in\mathbb{N}}\bigcap\limits_{k\geq m}D_{n,k}
\end{equation}
where $D_{n,k}$ is the set of elements of $\Gamma_{\lambda}(\mathbb{S}^{1})$ whose distortion at level $k$ is less or equal to $n$, in more precise terms $D_{n,k}=d_{k}^{-1}([0,n])$. This sets are closed in the $C^{1}$-topology since for every $k\in\mathbb{N}$ the map $d_{k}$ is continuous and so we conclude also that the sets is $\bigcap\limits_{k\geq m}D_{n,k}$ are closed, hence $\mathcal{UB}$ is the complement of a   countable intersection of closed sets and so we conclude that $\mathcal{B}$ is a $G_{\delta}$ set, by proposition \ref{dense} we conclude it is residual. Now by \cite{8} we know that  $\Gamma_{\lambda,er}(\mathbb{S}^{1})$ is residual in $\Gamma_{\lambda}(\mathbb{S}^{1})$ and by proposition \ref{residual} we conclude that the intersection $\mathcal{B}\cap\Gamma_{\lambda,er}(\mathbb{S}^{1})$ is residual in $\Gamma_{\lambda,er}(\mathbb{S}^{1})$. This finishes the proof of the theorem.
\end{proof}

  \section{Proof of Proposition  \ref{prop:extension}}\label{sec:extension}
\begin{proof}
We will start by recalling one of the classical tools to show that a measure is invariant. Let $f\in E^1(S^1)$ and, for all $h\in L^1_{\lambda}(S^1)$ and  $\mu_{h} \coloneqq h\cdot \lambda$, we define the transfer operator associated to $f$ and acting on $L^1_{\lambda}(S^1)$ as
\begin{equation}
    Ph=\dfrac{d\big(f_{*}\mu_{h}\big)}{d\lambda}. 
\end{equation}
This operator can be interpreted as the density of the push-forward of measures in respect to Lebesgue. It is well known that the fixed points of $P$ corresponds to the densities of $f$-invariant measures and that the transfer operator for maps of degree $n$ has an explicit formula given by 
\begin{equation}\label{eq:transfer}
    Ph(x)=\sum\limits_{y\in f^{-1}(x)}\dfrac{h(y)}{f'(y)}.
\end{equation}
Let us now consider $f$ to be an expanding circle map of degree $n$, represented as a full branch map of the unit interval with $n$ branches $\lbrace f_{i}\rbrace_{1\leq i\leq n}$ defined on adjacent intervals $\lbrace I_{i}=[x^{-}_{i},x^{+}_{i}]\rbrace_{1\leq i\leq n}$. Let $h:[0,1]\to\mathbb{R}_{+}$ be an $L_{\lambda}^1$ function. $h$ is the density of an $f$-invariant measure if and only if the relation \ref{eq:transfer} is satisfied, which can be written as:

\begin{equation}\label{eq:transfer1}
h(x)=\sum\limits_{y\in f^{-1}(x)}\dfrac{h(y)}{f'(y)}=\sum\limits_{1\leq y_{i}\leq n}\dfrac{h(y_{i})}{f'(y_{i})}.
\end{equation},
where each $y_{i}$ represents the pre-image of $x$ by the $i$-th branch of $f$.
Now take $h$ to be identically equal to 1 i,e the density of Lebesgue measure, and suppose that $n-1$ branches are known, we want to show we can construct the missing $n$-th branch such that the resulting map is a circle expanding map preserving Lebesgue measure. Equation \ref{eq:transfer1} is equivalent to

\begin{equation}\label{eq:transfer2}
\dfrac{1}{f'(y_{i_{0}})}=1-\sum_{\substack{1\leq i\leq n\\ i\neq i_{0}}}\dfrac{1}{f'(y_{i})},
\end{equation}
where $i_{0}$ is the index of the missing branch. Notice that $x_{i}=f_{i}^{-1}(x)$, equation \ref{eq:transfer2} then becomes:

\begin{equation}
    \dfrac{1}{f'(f^{-1}_{i_{0}}(x))}=1-\sum_{\substack{1\leq i\leq n\\ i\neq i_{0}}}\dfrac{1}{f'(f^{-1}_{i}(x))},
\end{equation}
If condition \ref{eq:condition 1} holds then we obtain the following first order $ODE$:

\begin{equation}
f'_{i_{0}}(x)=\dfrac{1}{1-\sum_{\substack{1\leq i\leq n\\ i\neq i_{0}}}\dfrac{1}{f'(f^{-1}_{i}(f_{i_{0}(x)}))}},
\end{equation}
this is a continuous first order ODE defined on a compact  rectangular domain, by Peano's existence theorem, there must exists a maximal solution defined on $I_{i_{0}}$ with the initial condition that $f_{i_{0}}(x^{-}_{i_{0}})=0$. 

We will show that this solution is an expanding diffeomorphism of $I_{i_{0}}$ onto the interval $[0,1]$ and that along the other fixed branches it defines a circle map, we will also show the solution is unique using the dynamics since Peano's existence theorem fails to ensure existence under only a continuity assumption. 

First, notice that by (\ref{eq:condition 1}) we get $f'_{i_{0}}>1$, hence it remains only to prove its surjective, which means $f_{i_{0}}(x^{+}_{i_{0}})=1$, indeed, by contradiction, suppose that $f_{i_{0}}(x^{+}_{i_{0}})<1$ and let $I^{\star}=[0,f_{i_{0}}(x^{+}_{i_{0}})]$. Notice that on $I^{\star}$ equation $(12)$ becomes:

\begin{equation}
\sum\limits_{y\in f^{-1}(x)}\dfrac{1}{f'(y)}=1
\end{equation}
this yields:

\begin{equation}
\lambda(f^{-1}(I^{\star}))=\int_{I^{\star}}\sum\limits_{y\in f^{-1}(x)}\dfrac{1}{f'(y)}dx=\lambda(I^{\star}),
\end{equation}
On the other hand, since $f^{-1}_{i_{0}}(I^{\star})=\emptyset$ we obtain that
\begin{equation}
    \lambda(f^{-1}([0,1]\setminus I^{\star})<\lambda(f^{-1}([0,1]\setminus I^{\star}),
\end{equation}
which leads to the following contradiction

\begin{equation*}
    \lambda(f^{-1}([0,1]))=\lambda(f^{-1}(I^{\star}))+ \lambda(f^{-1}([0,1]\setminus I^{\star})<1
\end{equation*}

We finally obtain that $f$ defines a full branch map of the interval, since also we have that $(13)$ is satisfied on all the unit interval then Lebesgue measure is preserved.

It remains to show that $f$ represents a circle map, indeed we need to check that $f'(x^{-}_{i})=f'(x^{+}_{i})$ as well as $f'(0)=f'(1)$. This follows directly my the assumption on the derivative of the branches at the end points of the partition elements.
\end{proof}

\section{Proof of proposition \ref{dense}}\label{sec:dense}

We will prove that there exists a dense set of Lebesgue reserving maps of the circle with unbounded distortion in $\Gamma_{\lambda}(\mathbb{S}^1)$. Our idea is to take an element of  $\Gamma_{\Lambda}(\mathbb{S}^1)$ with bounded distortion and prove it can be approximated arbitrarily by ones whitch have unbounded distortion. We recall the following useful definitions.

\begin{definition}
    We say that a map $\omega:\mathbb{R}_{+}\to\mathbb{R}_{+}$ is a modulus of continuity if $\omega(0)=0$, is continuous and concave. 
\end{definition}

\begin{definition}
    We say that a modulus of continuity $\omega$ is Dini-integrable if the following condition holds:

    \begin{equation*}
        \int_{0}^{1}\dfrac{\omega(t)}{t}dt<\infty.
    \end{equation*}
\end{definition}

\begin{proof}
    
Let $f\in \Gamma_{\lambda}(\mathbb{S}^1)\backslash \mathcal{B}$ and Let $\epsilon>0$, on a small enough neighborhood $V_{0}$ of the unique fixed point of $f$ which we are assuming to be $0$ let $\Tilde{f}'|_{V_{0}}=f'+\epsilon\omega$ where $\omega$ is a non Dini-integrable modulus of continuity. Let us extend the frst branch in a way that $d_{1}(f_{1},\Tilde{f_{1}})\leq\epsilon$ while keeping the remaining $n-2$ unchanged, this is possible by normality of the circle.

\begin{lemma}

The $(n-1)$-branches of $f$ obtained after perturbing the first branch extend to a Lebesgue preserving map $\Tilde{f}$ of degree $n$ which is $\epsilon$-close to $f$.
\end{lemma}

\begin{proof}
    The extension to a Lebesgue preserving map of degree $n$ follows by proposition \ref{prop:extension}. To see that it is $\epsilon$-close to $f$ let us consider $f_{n}$ and $\Tilde{f_{n}}$ to be the last branches of both maps, since they preserve Lebesgue measure we have that for a every interval $I\subset [0,1]$:

    \begin{equation*}
        \lambda(I)=\sum\limits_{1\leq i\leq n}\lambda(f_{i}^{-1}(I))= \sum\limits_{1\leq i\leq n}\lambda(\Tilde{f}_{i}^{-1}(I))
    \end{equation*}

    By construction this is equivalent to

    \begin{equation}
        \lambda(f_{1}^{-1}(I))-\lambda(\Tilde{f}_{1}^{-1}(I))=\lambda(f_{n}^{-1}(I))-\lambda(\Tilde{f}_{n}^{-1}(I))
    \end{equation}

    since $| \lambda(f_{1}^{-1}(I))-\lambda(\Tilde{f}_{1}^{-1}(I))|\leq\epsilon$, by Lebesgue density theorem we obtain that $d_{1}(f_{n},\Tilde{f}_{n})\leq \epsilon$. This finishes the proof.
\end{proof}

To finish the proof, it remain to show that these perturbations yield element which have unbounded distortion.

\begin{lemma}
    For every $\epsilon>0$ and $f\in \Gamma_{\lambda}(\mathbb{S}^1)\backslash \mathcal{B}$, the perturbed map $\Tilde{f}$ has unbounded distortion.
\end{lemma}

\begin{proof}
    By the formula given in the introduction of the definition of bunded distortion we have

    \begin{equation*}
|\log\dfrac{(\Tilde{f}^k)'(x)}{(\Tilde{f}^k)'(y)}|=|\sum\limits_{0\leq i\leq k-1} (\log (\Tilde{f}'(\Tilde{f}^i(x))-\log \Tilde{f}'(\Tilde{f}^i(y))|
\end{equation*}

\noindent Using mean value theorem, for every $0\leq i\leq k-1$ there exists 
\[
\lambda=\min\limits_{x\in S^1}\leq z_i\leq \sigma=\max\limits_{x\in S^1} |\Tilde{f}'(x)|>1
\]
 such that
\begin{equation*}
    \sum\limits_{0\leq i\leq k-1} \log \Tilde{f}'(\Tilde{f}^i(x))-\log \Tilde{f}'(\Tilde{f}^i(y))= \sum\limits_{0\leq i\leq k-1}\dfrac{1}{z_{i}}(\Tilde{f}'(\Tilde{f}^i(x))-\Tilde{f}'(\Tilde{f}^i(y))).
\end{equation*}
 Now for every $k\in\mathbb{N}$, let us take the first partition element of order $k$, i.e.  $\omega^{k}_{1}=[0,r_{k}]$ where $\sigma^{-k}\leq r_k\leq\lambda^{-k}$. Let us take $y=0$ and $x_{k}\in\omega^{k}_{1}$ such that $\Tilde{f}^{k}(x_k)\in V_{0}\backslash\lbrace 0\rbrace$, this possible by taking the pre-image of a point in $V_{0}\backslash\lbrace 0\rbrace$ by the first branch $\Tilde{f}_{1}^k$, that is, we consider $x_{k}=f^{-k}_{1}(x_0)$ We get

 \begin{equation*}
    \left|\log\dfrac{(\Tilde{f}^k)'(x_k)}{(\Tilde{f}^k)'(0)}\right|\geq\dfrac{1}{\sigma}\vert \sum\limits_{0\leq i\leq k-1} (\Tilde{f}'(\Tilde{f}^i(x_{k}))-\Tilde{f}'(0))\vert
\end{equation*}
but since $\Tilde{f}'(\Tilde{f}^i(x_{k})=f'(\Tilde{f}^i(x_{k})+\omega(\Tilde{f}^i(x_{k})$ we obtain

\begin{equation}
     \left|\log\dfrac{(\Tilde{f}^k)'(x_k)}{(\Tilde{f}^k)'(0)}\right|\geq\dfrac{1}{\sigma}\vert\sum\limits_{0\leq i\leq k-1} (f'(\Tilde{f}^i(x_{k}))-f'(0))+ \sum\limits_{0\leq i\leq k-1}\omega(\Tilde{f}^i(x_k))\vert
\end{equation}

by choice of $x_k$ we obtain that

\begin{equation*}
      \left|\log\dfrac{(\Tilde{f}^k)'(x_k)}{(\Tilde{f}^k)'(0)}\right|\geq \vert \sum\limits_{0\leq i\leq k-1} (f'(C\sigma^{i-k})-f'(0))+\sum\limits_{0\leq i\leq k-1} \omega(C\sigma^{i-k})\vert,
\end{equation*}
Where $C>0$ is a constant. Since $f$ has bounded distortion the first term is bounded, and since $\omega$ is not Dini-integrable by \cite{9} we deduce that the second sum diverges hence we obtain unbounded distortion.
\end{proof}
This finishes the proof of the proposition.
\end{proof}

\section{Proof of proposition \ref{residual}}\label{sec:residual}
We recall the definition of a residual set.
\begin{definition}
    A subset $R\subset X$ of a metric space is said to be residual if it is a dense $G_{\delta}$ set, we say that elements of $R$ are generic in $X$.
\end{definition}

\begin{proof}
We want to prove that the space $\Gamma_{\lambda}(\mathbb{S}^1)$ is residual in its completion. Its clear that the completion is the following space:
\begin{equation*}
    \Gamma_{\lambda}^{\star}(\mathbb{S}^1)=\lbrace f:\mathbb{S}^1\to \mathbb{S}^1\ \text{such that}\ f'(x)\geq 1\ \text{for all}\ x\in [0,1] \rbrace
\end{equation*}

Now consider a countable basis of the topology of $\mathbb{S}^1$ by closed intervals $\lbrace I_{n}\rbrace$ and define the set:

\begin{equation*}
    S_{I_{n}}=\lbrace f\in   \Gamma_{\lambda}^{\star}(\mathbb{S}^1)\ \text{such that}\ f'|_{I_{n}}>1  \rbrace.
\end{equation*}

Clearly we have:
\begin{equation*}
    \Gamma_{\lambda}(\mathbb{S}^1)=\bigcap\limits_{n\in\mathbb{N}} S_{I_{n}}
\end{equation*}
and that $ S_{I_{n}}$ are open sets in the $C^1$ topology and hence  $\Gamma_{\lambda}(\mathbb{S}^1)$ is a $G_{\delta}$ set, on the other hand, every element in $ \Gamma_{\lambda}^{\star}(\mathbb{S}^1)$ is clearly arbitrarily close to an element of $\Gamma_{\lambda}(\mathbb{S}^1)$ and hence $\Gamma_{\lambda}(\mathbb{S}^1)$ is a residual set of its completion.
\end{proof}

\begin{center}
\textbf{Acknowledgments.}
\end{center}

\noindent I am very thankful to Stefano Luzzatto for reading this note and for his useful comments and suggestion in inmproving the text.

\printbibliography

\end{document}